\documentclass[12pt]{amsart}

\usepackage{amsmath}
\usepackage{amssymb}
\usepackage{ascmac}
\usepackage{amsthm}

\makeatletter

\@addtoreset{equation}{section}
\makeatother 

\theoremstyle{plain}
\newtheorem{thm}{Theorem}[section]
\newtheorem*{thm*}{Theorem}
\newtheorem{prop}[thm]{Proposition}

\theoremstyle{remark}

\newcommand{\vol}{\operatorname{vol}}

\newcommand{\diam}{\operatorname{Diam}}

\usepackage{latexsym}

\title[Two new universal inequalities]{Two new universal inequalities for Neumann eigenvalues of the Laplacian on a planar convex domain}

\author{Kei Funano}
\address{Division of Mathematics \& Research Center for Pure and Applied Mathematics, Graduate School of Information Sciences, Tohoku University, 6-3-09 Aramaki-Aza-Aoba, Aoba-ku, Sendai 980-8579, Japan}
\email{kfunano@tohoku.ac.jp}

\date{\today}

\begin{document}
\maketitle

\begin{abstract}
We establish two new universal inequalities for Neumann eigenvalues of the Laplacian on a planar convex domain.

\medskip

\noindent\textsc{R\'esum\'e.} Nous \'etablissons deux nouvelles in\'egalit\'es universelles pour les valeurs propres de Neumann du laplacien sur un domaine convexe plan.
\end{abstract}

%%%%%%%%%%%%%%%%%%%%%%%%%%%%
%%%%%%%%%%%%%%%%%%%%%%%%%%%%
%%%%%%%%%%%%%%%%%%%%%%%%%%%%

%%%%%%%%%%%%%%%%%%%%%%%%%%%
%%%%%%%%%%%%%%%%%%%%%%%%%%%
%%%%%%%%%%%%%%%%%%%%%%%%%%%

%%%%%%%%%%%%%%%%%%%%%%%%

%%%%%%%%%%%%%%%%%%%%%%%%%%%%
%%%%%%%%%%%%%%%%%%%%%%%%%%%%
%%%%%%%%%%%%%%%%%%%%%%%%%%%%

%%%%%%%%%%%%%%%%%%%%%%%%

%%%%%%%%%%%%%%%%%%%%%%%%%%%%
%%%%%%%%%%%%%%%%%%%%%%%%%%%%
%%%%%%%%%%%%%%%%%%%%%%%%%%%%
\section{Introduction}
Let $\Omega$ be a bounded domain in $\mathbb{R}^n$ with Lipschitz boundary. We consider the following Neumann eigenvalue problem of the Laplacian: 
\begin{equation}\label{Neumann eigenvalue problem}
\begin{cases}
-\Delta u = \mu u \quad \text{in } \Omega,\\
\partial_n u=0\quad \text{on }\partial \Omega,
\end{cases}
\end{equation}
where $\partial_n$ denotes the directional derivative with respect to $n$, the outward unit normal vector to $\partial\Omega$.
In this situation the compactness of the embedding $H^1(\Omega) \subset L^2(\Omega)$ implies that the eigenvalues of (\ref{Neumann eigenvalue problem}) can be listed with multiplicity as follows: 
\begin{align*}
     0=\mu_0(\Omega)<\mu_1(\Omega)\le\mu_2(\Omega)\le \dots  \le \mu_k(\Omega)\le \mu_{k+1}(\Omega)\le \dots \to\infty.
\end{align*}

Our main concern in this paper is the relation between those eigenvalues. One of the main theorems in this paper is the following. We denote by $\alpha\lesssim \beta$ if $\alpha\leq c\beta$ for some universal concrete constant $c>0$ (which means $c$ does not depend on any parameters such as dimension and $k$ etc).
\begin{thm}\label{MTHM}Let $\Omega$ be a planar bounded convex domain. Then for any $k\geq l$ we have
    \begin{align*}
        \mu_k(\Omega)\lesssim \Big(\frac{k}{l}\Big)^{2}\mu_l(\Omega).
    \end{align*}
\end{thm}

Let us mention the related known inequalities between the eigenvalues.
In \cite[Theorem 1.1]{Liu}, Liu showed the sharp dimension-free universal inequality $\mu_k(\Omega)\lesssim k^2\mu_1(\Omega)$ for any bounded convex domain $\Omega$ in $\mathbb{R}^n$.
In \cite[Theorem 1.2]{F2} the author proved that 
\begin{align}\label{univ}
    \mu_{k+1}(\Omega)\lesssim n^4\mu_k(\Omega)
\end{align}under the same situation of that of Liu's. The above theorem unifies the Liu inequality and the inequality (\ref{univ}) in the planar case.

We also obtain a reverse inequality: 
\begin{thm}\label{MTHM2}
    Let $\Omega$ be a planar bounded convex domain. Then for any $k\geq l$ we have
    \begin{align*}
        \mu_k(\Omega)\gtrsim \frac{k}{l} \mu_l(\Omega).
    \end{align*}
\end{thm}

In \cite[Theorem 1.1 and Remark 1.1]{F3} building on theorems by Milman (\cite{Mi}) and Cheng-Li (\cite{CL}) the author showed the sharp universal inequality $\mu_k(\Omega)\gtrsim k^{\frac{2}{n}}\mu_1(\Omega)$ for any bounded convex domain $\Omega$ in $\mathbb{R}^n$. The above theorem generalizes this inequality in the planar case. 

In the proof of the above main theorems we use a variant of domain monotonicity for eigenvalues to reduce the problem to the case of rectangular domains. In this aspect we give a new brief proof for a weaker version of the inequality (\ref{univ}) in the last section (see Theorem \ref{Weaker ineq}). It is weaker but its proof is new and rather shorter than the proof of the inequality (\ref{univ}) so the author includes it in this paper.
\section{Preliminaries}

We recall several results which will be used in the proof of main theorems.

In the following proposition Hatcher showed that bounded convex domains can be approximated by rectangular domains in some sense. He used the John theorem which asserts that bounded convex domains can be approximated by ellipsoids (\cite[Theorem III]{J}). 

\begin{prop}[{\cite[Proposition 2.3]{H}}]\label{Hacher lem} Let $\Omega$ be a bounded convex domain in $\mathbb{R}^n$. Choosing the coordinate axes appropriately if necessary we can find a sequence $a_1,a_2,\cdots, a_n$ of positive real numbers such that  
    \begin{align*}
       \Big[-\frac{a_1}{\sqrt{n}},\frac{a_1}{\sqrt{n}}\Big]\times \cdots \times \Big[-\frac{a_n}{\sqrt{n}},\frac{a_n}{\sqrt{n}}\Big]\subseteq  \Omega\subseteq [-a_1 n, a_1 n]\times \cdots [-a_n n, a_n n].
    \end{align*}
\end{prop}

Combining the above proposition with the next theorem we see that the proof of the main theorem reduces to the case where the domain is rectanglular.
It is a variant of domain monotonicity, it states that the smaller domain has a larger eigenvalue up to a multiplicative constant factor under convexity assumption.
\begin{thm}[{\cite[Theorem 1.1]{F}}]\label{domain monotonicity}For any two
bounded convex domains $\Omega\subseteq \Omega'$ in $\mathbb{R}^n$ their Neumann eigenvalues of the Laplacian satisfy
    \begin{align*}
        \mu_k(\Omega')\lesssim n^2\mu_k(\Omega).
    \end{align*}
\end{thm}

For a convex domain, Neumann eigenvalues of the Laplacian are related with the diameter as follows:
\begin{thm}[{Payne–Weinberger, \cite[(1.9)]{PW}}]\label{Diamthm}Suppose that $\Omega$ is a bounded convex domain in $\mathbb{R}^n$. Then we have $\mu_1(\Omega)\geq \frac{\pi^2}{(\diam \Omega)^2}$.
\end{thm}

In the following theorem $j_{\nu,k}$ denotes the $k$-th positive zero of the Bessel function $J_{\nu}$ of order $\nu$.
\begin{thm}[{Kr\"oger, \cite[Theorem 1]{Kr}}]\label{Kr ineq}
    Let $\Omega$ be a bounded convex domain in $\mathbb{R}^n$. Assume that $n>2$. Then we have
    \begin{enumerate}
        \item $(\diam \Omega)^2\mu_k(\Omega)\leq 4j_{\frac{n-2}{2},\frac{k+1}{2}}^2 $ if $k$ is odd, and
        \item $(\diam \Omega)^2\mu_k(\Omega)\leq (j_{\frac{n-2}{2},\frac{k}{2}}+j_{\frac{n-2}{2},\frac{k+2}{2}})^2 $ if $k$ is even.
    \end{enumerate}
    Now assume that $n=2$. Then we have
    \begin{align*}
        (\diam \Omega)^2\mu_k(\Omega)\leq (2j_{0,1}+(k-1)\pi)^2\lesssim k^2.
    \end{align*}
\end{thm}
Kr\"{o}ger proved the following upper bound for Neumann eigenvalues in terms of volume using harmonic analysis. 
\begin{thm}[{Kr\"{o}ger, \cite[Corollary 2]{Kr2}}]\label{Kr ineq2}Let $\Omega$ be a (not necessarily convex) bounded domain in $\mathbb{R}^n$ with Lipschitz boundary. Then for any natural number $k$ we have
\begin{align*}
    \mu_k(\Omega)\leq (2\pi)^2\Big(\frac{n+2}{2}\Big)^{\frac{2}{n}}\Big(\frac{k}{\omega_n \vol \Omega}\Big)^{\frac{2}{n}}\lesssim \Big(\frac{k}{\omega_n \vol \Omega}\Big)^{\frac{2}{n}},
\end{align*}where $\omega_n$ is the volume of a unit ball in $\mathbb{R}^n$.
\end{thm}

Let $\Omega$ be a bounded domain in a Euclidean space and $\{\Omega_i\}_{i=1}^{k}$ be a finite partition of $\Omega$ by subdomains;
    $\Omega=\bigcup_i \Omega_i$ and $\vol (\Omega_i\cap \Omega_j)=0$ for different $i\neq j$. We need the following Buser theorem to give a lower bound of eigenvalues of the Laplacian.

\begin{thm}[{Buser, \cite[Theorem 8.2.1]{B}}]\label{Buserthm}Under the above situation we have
    \begin{align*}
        \mu_k(\Omega)\geq \min_i \mu_1(\Omega_i).
    \end{align*}
\end{thm}

\section{Proof of the Main Theorems}
In this section we prove the main theorems.
\begin{proof}[Proof of Theorem \ref{MTHM}]
    By virtue of Proposition \ref{Hacher lem} and Theorem \ref{domain monotonicity} we may assume that our $\Omega$ is a rectangular region, say $\Omega=[-a,a]\times [-b,b]$. Without loss of generality we may also assume that $a\leq b$.

    Put $R:=\frac{Ck}{l\sqrt{\mu_k(\Omega)}}$, where $C>0$ is a sufficiently large constant which will be specified later. We want to show that $\Omega$ can be decomposed into $l'$ convex subdomains $\{\Omega_i\}_{i=1}^{l'}$ such that $\diam (\Omega_i)\lesssim R$ and $l'\leq l$. If this holds, then using Buser's theorem (Theorem \ref{Buserthm}) and the Payne-Weinberger inequality (Theorem \ref{Diamthm}) we obtain
    \begin{align*}
        \mu_l(\Omega)\geq \mu_{l'}(\Omega)\geq \min_i\mu_1(\Omega_i)\gtrsim \frac{1}{R^2} = \frac{l^2\mu_k(\Omega)}{C^2k^2},
    \end{align*}which completes the proof of the theorem.

    We first consider the case where $a\leq 2R$. In this case Kr\"{o}ger's inequality (Theorem \ref{Kr ineq}) gives 
    \begin{align*}
        R\gtrsim \frac{Ck}{l \frac{k}{\diam \Omega}}=\frac{C\diam \Omega}{l}\geq \frac{Cb}{l}.
    \end{align*}We can take an absolute constant $C_1$ such that if $C\geq C_1$, then we have $R\geq 2b/l$. We thereby have a partition 
    \begin{align*}
        [-b,b]=[b_1,b_2]\cup[b_2,b_3]\cup \cdots\cup [b_{l'},b_{l'+1}],
    \end{align*}such that $l'\leq l$ and $b_{i+1}-b_i\leq R$. Thus setting $\Omega_i:= [-a,a]\times [b_i,b_{i+1}]$ we get the desired convex partition of $\Omega$; $\Omega=\bigcup_{i=1}^{l'}\Omega_i$, where $\diam (\Omega_i)\lesssim R$.

    Let us consider the case where $a>2R$. We put 
    \begin{align*}
        \Omega':=\{x\in \Omega \mid d(x,\partial \Omega)\geq R\},
    \end{align*}where 
    \begin{align*}
        d(x,\partial \Omega):=\inf \{|x-y| \mid y\in \partial \Omega\}.
    \end{align*}
    Note that $\Omega'$ is a nonempty (noncollapsed) rectangular region. Let $\{x_i\}_{i=1}^{l'}$ be a maximal set of $2R$-separated points in $\Omega'$. Since $\{B(x_i,R)\}_{i=1}^{l'}$ are disjoint balls in $\Omega$, we have
    \begin{align*}
       \frac{l'\omega_2(Ck)^2}{l^2\mu_k(\Omega)}= l'\omega_2 R^2 =\sum_{i=1}^{l'} \vol B(x_i,R)\leq \vol \Omega.
    \end{align*}
    Applying Theorem \ref{Kr ineq2} to the above inequality we have
    \begin{align*}
        l'\lesssim \frac{l}{k}\cdot \frac{l}{C^2}\leq \frac{l}{C^2}.
    \end{align*}
    There exists an absolute constant $C_2>0$ such that if $C\geq C_2$ we then obtain $l'\leq l$. 

    The maximality of the points $\{x_i\}_{i=1}^{l'}$ gives $\Omega'\subseteq \bigcup_{i=1}^{l'}B(x_i,2R)$. Also observe that the $\sqrt{2}R$-neighbourhood of $\Omega'$ covers $\Omega$ (note that the $R$-neighbourhood of $\Omega'$ does not cover $\Omega$). We thereby have $\Omega\subseteq \bigcup_{i=1}^{l'}B(x_i,(2+\sqrt{2})R)$. Therefore if we introduce the \emph{Voronoi cells}
    \begin{align*}
        \Omega_i:=\{x\in \Omega \mid |x-x_i|\leq |x-x_j| \text{ for all }j\neq i\},
    \end{align*}then $\{\Omega_i\}_{i=1}^{l'}$ is a convex partition of $\Omega$. Since $\Omega_i\subseteq B(x_i,(2+\sqrt{2})R)$ we obtain the desired convex partition $\{\Omega_i\}_{i=1}^{l'}$ of $\Omega$. 
    
    Taking $C$ as the maximum of $C_1$ and $C_2$ completes the proof of the theorem.
\end{proof}
The same method in the above proof applies to the proof of Theorem \ref{MTHM2} as follows:
\begin{proof}[Proof of Theorem \ref{MTHM2}]
Again we may assume that our domain $\Omega$ is rectangular, say $\Omega=[-a,a]\times [-b,b]$ and $a\leq b$. 

Put $R:=\frac{C\sqrt{l}}{\sqrt{k\mu_l(\Omega)}}$, where $C>0$ is a large constant which will be specified later. Using the same reasoning as in the proof of Theorem \ref{MTHM} it suffices to prove the existence of a convex partition $\{\Omega_i\}_{i=1}^{k'}$ of $\Omega$ such that $k'\leq k$ and $\diam (\Omega_i)\lesssim R$ for each $i$. 

The proof of the existence is divided into two cases. The first case is the case where $a\leq 2R$. In this case by Kr\"{o}ger's inequality (Theorem \ref{Kr ineq}) we have
\begin{align*}
    R\gtrsim \frac{C \sqrt{l}}{\sqrt{k \frac{l^2}{(\diam \Omega)^2}}}\geq \frac{Cb}{\sqrt{l k}} \geq \frac{Cb}{k}.
\end{align*}There exists an absolute constant $C_1>0$ such that if $C\geq C_1$ then we get $R\geq \frac{2b}{k}$. Therefore the same argument as in the Theorem \ref{MTHM} shows the existence of the desired convex partition. 

The second case is the case where $a>2R$. In this case as in the proof of Theorem \ref{MTHM} we consider the rectangular subdomain
\begin{align*}
    \Omega':=\{x\in \Omega \mid d(x,\partial\Omega)\geq R\}
\end{align*}and take a maximal set of $2R$-separated points $\{x_i\}_{i=1}^{k'}$ in $\Omega'$. If we show $k'\leq k$ then as in the proof of Theorem \ref{MTHM} we get the desired convex partition $\{\Omega_i\}_{i=1}^{k'}$. Since $\{B(x_i,R)\}_{i=1}^{k'}$ are disjoint balls in $\Omega$, we have
\begin{align*}
k'\omega_2R^2=\sum_{i=1}^{k'}\vol B(x_i,R)\leq \vol \Omega.   
\end{align*}Applying Kr\"{o}ger's inequality (Theorem \ref{Kr ineq2}) to the above inequality we obtain
\begin{align*}
    k'\leq \frac{\vol \Omega}{w_2 R^2}=\frac{k\mu_l(\Omega) \vol \Omega}{C^2\omega_2 l}\lesssim \frac{k}{C^2}.
\end{align*}We can take an absolute constant $C_2>0$ such that if $C\geq C_2$ we then obtain $k'\leq k$.

Setting $C:=\max\{C_1,C_2\}$ completes the proof.
\end{proof}
\section{An weaker inequality of (\ref{univ})}
In this section we give a short proof of the following:
\begin{thm}\label{Weaker ineq}Let $\Omega$ be a bounded convex domain in $\mathbb{R}^n$. Then we have
\begin{align*}
    \mu_{k+1}(\Omega)\lesssim n^7\mu_k(\Omega)
\end{align*}for any natural number $k$.
\begin{proof}Proposition \ref{Hacher lem} gives the existence of a rectangular domain $R=[-a_1,a_1]\times[-a_2,a_2]\times \cdots \times [-a_n,a_n]$ such that
      \begin{align*}
     \frac{1}{\sqrt{n}}R\subseteq  \Omega\subseteq nR.
    \end{align*}We claim that $\mu_{k+1}(R)\lesssim \mu_k(R)$. If this claim holds, applying Theorem \ref{domain monotonicity} twice, we then obtain
    \begin{align*}
        \mu_{k+1}(\Omega)\lesssim n^2\mu_{k+1}\Big(\frac{1}{\sqrt{n}}R\Big)=n^3\mu_{k+1}(R)\lesssim n^3\mu_k(R)=\ &n^5\mu_k(n R)\\ \lesssim \ & n^7\mu_k(\Omega),
    \end{align*}which implies the conclusion of the theorem.

    To prove the claim, let $\widetilde{R}$ be the flat torus whose fundamental domain is $R$. Recall that Neumann eigenvalues of the Laplacian on $R$ are of the form
    \begin{align*}
        \frac{\pi^2}{4}\Big\{\frac{k_1^2}{a_1^2}+\frac{k_2^2}{a_2^2}+\cdots + \frac{k_n^2}{a_n^2}\Big\}, \ k_1,k_2,\cdots, k_n\in \mathbb{N}\cup \{0\}
    \end{align*}and eigenvalues $\lambda_k(\widetilde{R})$ of the Laplacian on $\widetilde{R}$ are of the form
    \begin{align*}
        4\pi^2\Big\{\frac{l_1^2}{a_1^2}+\frac{l_2^2}{a_2^2}+\cdots + \frac{l_n^2}{a_n^2}\Big\}, \ l_1,l_2,\cdots, l_n\in \mathbb{Z}
    \end{align*}(\cite[Section II.2]{Ch}). Suppose that 
    \begin{align*}
        \mu_k(R)=\frac{\pi^2}{4}\Big\{\frac{m_1^2}{a_1^2}+\cdots +\frac{m_n^2}{a_n^2}\Big\} \text{ and }\mu_{k+1}(R)= \frac{\pi^2}{4}\Big\{\frac{m_1'^2}{a_1^2}+\cdots +\frac{m_n'^2}{a_n^2}\Big\}
    \end{align*}for some $m_1,\cdots,m_n,m_1',\cdots,m_n'\in \mathbb{N}\cup \{0\}$. Then we can take $l$ such that 
    \begin{align*}
        \lambda_l(\widetilde{R})=4\pi^2\Big\{\frac{m_1^2}{a_1^2}+\cdots +\frac{m_n^2}{a_n^2}\Big\} \text{ and }\lambda_{l+1}(\widetilde{R})= 4\pi^2\Big\{\frac{m_1'^2}{a_1^2}+\cdots +\frac{m_n'^2}{a_n^2}\Big\}.
    \end{align*}Since $\widetilde{R}$ is a homogeneous manifold, we use P.~Li's theorem (\cite[Theorem 11]{Li}) to get $\lambda_{l+1}(\widetilde{R})\lesssim \lambda_l(\widetilde{R})$, which implies $\mu_{k+1}(R)/\mu_k(R)=\lambda_{l+1}(\widetilde{R})/\lambda_l(\widetilde{R})\lesssim 1$. This completes the proof.
\end{proof}
\end{thm}
\emph{Acknowledgement.} The author would like to thank Arataka Funano and Chizu Funano for their massive support. The author also thanks the anonymous referee for his or her suggestion. This work was supported by JSPS KAKENHI Grant Number JP24K06731.

%%%%%%%%%%%%%%%%%%%%%%%%%%%%

\end{document}